\documentclass[11pt,a4paper]{amsart}
\usepackage{graphicx,amscd,color}

\theoremstyle{plain}
\newtheorem{theorem}{Theorem}[section]

\newtheorem{question}[theorem]{Question}

\theoremstyle{remark}

\newtheorem*{claim}{Claim}

\begin{document}
\title [Topologically minimal surfaces]
{On topologically minimal surfaces of high genus}

\author[J. H. Lee]{Jung Hoon Lee}
\address{Department of Mathematics and Institute of Pure and Applied Mathematics,
Chonbuk National University, Jeonju 561-756, Korea}
\email{junghoon@jbnu.ac.kr}

\maketitle
\begin{abstract}
We show that a $3$-manifold containing an incompressible surface has
topologically minimal surfaces of arbitrary high genus.
\end{abstract}

\section{Introduction}

Let $S$ be a closed connected orientable surface in a $3$-manifold $M$.
Define the {\em disk complex} $\mathcal{D}_S$ as follows.
\begin{itemize}
\item Vertices of $\mathcal{D}_S$ are isotopy classes of compressing disks for $S$.
\item A collection of $k+1$ distinct vertices constitute a $k$-simplex if
      there are pairwise disjoint representatives.
\end{itemize}

A surface $S$ is {\em topologically minimal}
if $\mathcal{D}_S$ is empty or $\pi_i(\mathcal{D}_S)$ is non-trivial for some $i$.
A {\em topological index} of $S$ is $0$ if $\mathcal{D}_S$ is empty,
and the smallest $n$ such that $\pi_{n-1}(\mathcal{D}_S)$ is non-trivial, otherwise.

Topologically minimal surfaces are defined by Bachman
as topological analogues of geometrically minimal surfaces \cite{Bachman3}.
They generalize well-known classes of surfaces such as incompressible surfaces,
strongly irreducible surfaces and critical surfaces \cite{Bachman1}, \cite{Bachman2},
and share nice properties.
For example, if a $3$-manifold contains a topologically minimal surface and
an incompressible surface, then the two surfaces can be isotoped so that
any intersection loop is essential on both surfaces.

In this paper, we show that any $3$-manifold containing an incompressible surface
has topologically minimal surfaces of arbitrary high genus.
This gives a negative answer to the conjecture \cite[Conjecture 5.7]{Bachman3}
that the only connected topologically minimal surfaces in $\textrm{(surface)} \times I$
are a single copy of the surface and two copies of the surface connected by an unknotted tube.

\section{Retraction of disk complex onto $S^{n-1}$}

In this section we consider a sufficient condition for a surface to be topologically minimal.
The idea is due to \cite[Claim 5]{Bachman-Johnson}.

Let $S$ be a surface in a $3$-manifold and
$D_0$ and $E_0$ be compressing disks for $S$ that intersect.
Then $D_0\cup E_0$ spans a subcomplex $S^0$ of $\mathcal{D}_S$.
Suppose that $D_1$ and $E_1$ are compressing disks for $S$ disjoint from $D_0\cup E_0$
and $D_1$ intersects $E_1$.
Then $S^0\cup (D_1\cup E_1)$ spans a subcomplex $S^1$ of $\mathcal{D}_S$.
The subcomplex $S^1$ can be thought as a suspension of $S^0$ over $D_1\cup E_1$.
If $D_2$ and $E_2$ are compressing disks for $S$ disjoint from $S^1$ and
$D_2$ intersects $E_2$,
then $S^1\cup (D_2\cup E_2)$ spans a $2$-sphere $S^2$
which is a suspension of $S^1$ over $D_2\cup E_2$. See Figure \ref{fig1}.
In this way, we can consider a sequence of embedded spheres
$S^0\subset S^1\subset\cdots\subset S^{n-1}$ in $\mathcal{D}_S$
obtained by iterated suspensions.

\begin{figure}[ht!]
\begin{center}
\includegraphics[width=10cm]{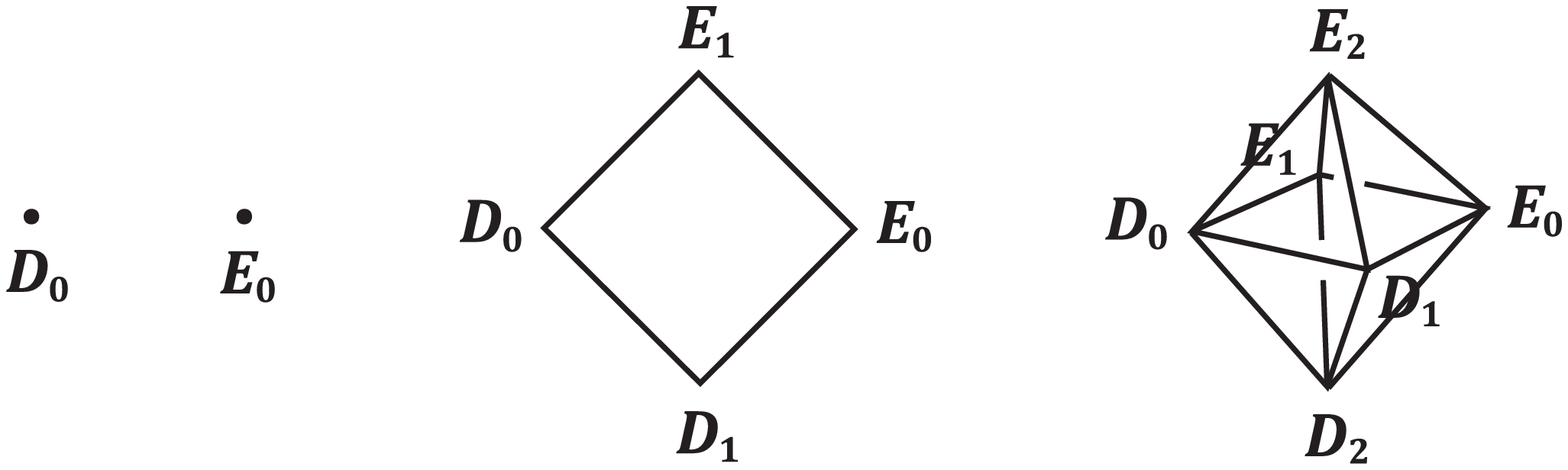}
\caption{}\label{fig1}
\end{center}
\end{figure}

Suppose that there exists a retraction $r:\mathcal{D}_S\rightarrow S^{n-1}$
of $\mathcal{D}_S$ onto a subcomplex $S^{n-1}$.
Let $i:S^{n-1}\rightarrow \mathcal{D}_S$ be the inclusion.
Then the induced maps on $\pi_{n-1}$ satisfy $r_*\circ i_*=\mathrm{id}_*(\pi_{n-1}(S^{n-1}))$.
Therefore $r_*:\pi_{n-1}(\mathcal{D}_S)\rightarrow \pi_{n-1}(S^{n-1})$
is a non-zero map and $\pi_{n-1}(\mathcal{D}_S)$ is non-trivial.
Hence $S$ is a topologically minimal surface of index at most $n$.

\section{Construction of topologically minimal surfaces}

Let $F$ be a genus $g$ incompressible surface in a $3$--manifold $M$.
Then $F$ is an index $0$ topologically minimal surface.

Take two parallel copies of $F$ and connect them by an unknotted tube, and
let $F_1$ be the resulting surface.
Let $V$ denote one side of $F_1$ which is homeomorphic to $\textrm{(punctured surface)} \times I$
and $W$ denote the other side of $F_1$.
Then there is a unique compressing disk $E$ for $F_1$ in $W$ and
any compressing disk in $V$ intersects $E$.
Hence $F_1$ is a strongly irreducible surface,
and equivalently a surface of topological index $1$.
So we can define a retraction $r:\mathcal{D}_{F_1}\rightarrow S^0$,
which sends all compressing disks in $V$ to some fixed compressing disk $D$ in $V$
and sends $E$ to $E$.
See Figure \ref{fig2}.

\begin{figure}[ht!]
\begin{center}
\includegraphics[width=6cm]{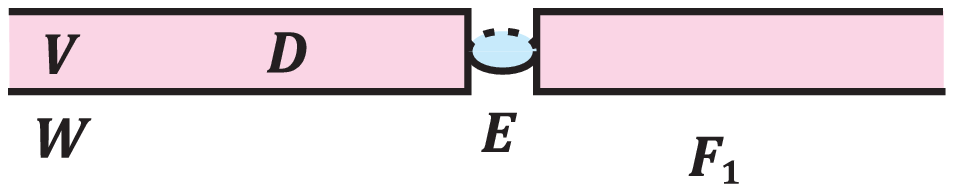}
\caption{}\label{fig2}
\end{center}
\end{figure}

We will inductively construct a genus $(n+2)g$ topologically minimal surface from
a genus $(n+1)g$ topologically minimal surface as follows.
Let $F_n$ be a connected surface obtained from $n+1$ parallel copies of $F$ by
connecting each pair of two consecutive parallel copies along an unknotted tube.
We designate one side of $F_n$ as $V$ and the other side of $F_n$ as $W$.
Even if $F_n$ is non-separating in $M$, we say that a compressing disk is in $V$ or $W$
if a small collar neighborhood of the boundary of the disk is in $V$ or $W$ respectively.
Suppose that there exists a retraction $r:\mathcal{D}_{F_n}\rightarrow S^{n-1}$,
where $S^{n-1}$ is a subcomplex of $\mathcal{D}_{F_n}$
obtained from $S^0$ by iterated suspensions as in Section $2$.
Let $F_{n+1}$ be a connected surface obtained from $F_n$ by
adding one more copy of $F$ and tubing along an unknotted tube.
Note that every compressing disk for $F_n$, in particular every compressing disk of $S^{n-1}$,
can be regarded as a compressing disk for $F_{n+1}$ also.
Let $V'$ and $W'$ denote the same side as $V$ and $W$ respectively.
Let $E$ be the unique compressing disk for the lastly attached tube in $W'$.
Let $D$ be a fixed compressing disk in $V'$ that intersects $E$ and
disjoint from every compressing disk of $S^{n-1}$,
e.g. a compressing disk that lies in $\textrm{(punctured surface)} \times I$ 
intersecting $E$ in two points.
See Figure \ref{fig3}.

\begin{figure}[ht!]
\begin{center}
\includegraphics[width=12.5cm]{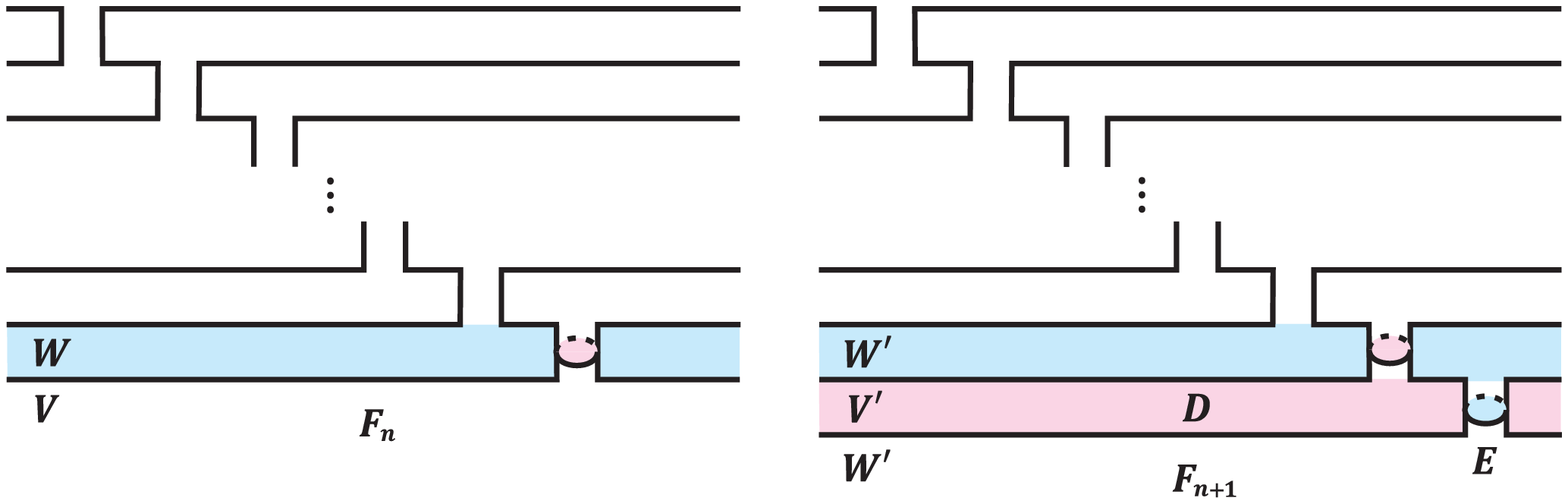}
\caption{}\label{fig3}
\end{center}
\end{figure}

Compressing disks in $\mathcal{D}_{F_{n+1}}$ can be partitioned into the following four types.

\begin{enumerate}
\item $E$
\item compressing disks in $W'$ other than $E$
\item compressing disks in $V'$ that intersect $E$
\item compressing disks in $V'$ that are disjoint from $E$.
\end{enumerate}

Let $E_0$ be a compressing disk that belongs to type $(2)$,
and assume that $E_0$ intersects $E$ transversely and minimally.
We may assume that there is no circle component of intersection in $E_0\cap E$.
If $E_0\cap E$ is nonempty, let $\gamma$ be an outermost arc of $E_0\cap E$ in $E_0$
and $\Delta_0$ be the corresponding outermost disk of $E_0$ cut by $\gamma$.
Let $\Delta'$ and $\Delta''$ be the two disks of $E$ cut by $\gamma$.
By the minimality of $|E_0\cap E|$, both $\Delta_0\cup\Delta'$ and $\Delta_0\cup\Delta''$
are essential in $W'$.
After an isotopy, both $\Delta_0\cup\Delta'$ and $\Delta_0\cup\Delta''$
can be regarded as compressing disks for $F_n$ in $W$.
The two disks $\Delta_0\cup\Delta'$ and $\Delta_0\cup\Delta''$ are not isotopic in $W'$,
however they represent the same isotopy class in $W$.
Also in the case that $E_0\cap E$ is empty, $E_0$ can be
regarded as a compressing disk for $F_n$ in $W$.

Let $D_0$ be a disk that belongs to type $(4)$.
Then $D_0$ can be regarded as a compressing disk for $F_n$ in $V$.

Now we define a map $r':\mathcal{D}_{F_{n+1}}\rightarrow S^n$,
which extends $r:\mathcal{D}_{F_n}\rightarrow S^{n-1}$.
The disks $D$ and $E$ are disjoint from the compressing disks of $S^{n-1}$.
So $S^{n-1}\cup\{D, E\}$ spans a subcomplex $S^n$ of $\mathcal{D}_{F_{n+1}}$.
There is a sequence of embedded spheres
$S^0\subset S^1\subset\cdots\subset S^{n-1}\subset S^{n}$ in $S^n$.

Define $r'(E)$ to be $E$.
We observed that a disk $E_0$ belonging to type $(2)$ has an outermost disk (or $E_0$ itself),
which gives rise to a compressing disk $\Delta$ for $F_n$ in $W$.
Among all such disks $\Delta$'s of $E_0$, consider the smallest $i$ such that
$S^i$ contains an image by $r$ of some $\Delta$.
Suppose that there are two such disks $\Delta_1$ and $\Delta_2$ of $E_0$,
i.e. both $r(\Delta_1)$ and $r(\Delta_2)$ are contained in $S^i-S^{i-1}$.
There are only two vertices in $S^i-S^{i-1}$.
Since the two vertices of $S^i-S^{i-1}$ represent disks in opposite sides of $F_n$,
and $\Delta_1$ and $\Delta_2$ are in the same side of $F_n$,
we can see that $r(\Delta_1)=r(\Delta_2)$.
(In fact, we are assuming that inductively $r$ is defined in this manner.)
So if we define $r'(E_0)$ to be $r(\Delta_1)$, it is well-defined.

For a disk belonging to type $(3)$, we define its image by $r'$ to be $D$.
A disk $D_0$ belonging to type $(4)$ can be regarded as a compressing disk for $F_n$ in $V$.
We define $r'(D_0)$ to be $r(D_0)$.
So far, $r'$ is defined on vertices of $\mathcal{D}_{F_{n+1}}$.

\begin{claim}
The map $r'$ sends endpoints of an edge to endpoints of an edge or a single vertex.
\end{claim}

\begin{proof}
To prove the claim is equivalent to show that $r'$ sends two disjoint disks of $\mathcal{D}_{F_{n+1}}$
to disjoint disks of $S^n$ or to the same disk.
There are several cases to consider according to the types of the two disks.

Case $1$. One disk is $E$ and the other disk $K$ disjoint from $E$ is of type $(2)$ or $(4)$.

Since a disk of type $(2)$ or $(4)$ is mapped into $S^{n-1}$,
$r'(K)$ is disjoint from $E$.

Case $2$. One disk $D_0$ is of type $(3)$ and
the other disk $K$ disjoint from $D_0$ is of type $(2)$ or $(4)$.

Since $r'(D_0)=D$ and $r'(K)\in S^{n-1}$, they are disjoint.

Case $3$. Both of two disjoint disks $D_1$ and $D_2$ are of type $(3)$.

Both $r'(D_1)$ and $r'(D_2)$ are $D$.

Case $4$. One disk $E_0$ is of type $(2)$ and
the other disk $D_0$ disjoint from $E_0$ is of type $(4)$.

For a disk of type $(2)$, the map $r'$ is defined using its outermost subdisk.
Since $E_0$ and $D_0$ are disjoint, so are $r'(E_0)$ and $r'(D_0)$.

Case $5$. Both of two disjoint disks $E_1$ and $E_2$ are of type $(2)$.

Let $\Delta_1$ and $\Delta_2$ be compressing disks for $F_n$ in $W$
obtained from outermost disks of $E_1$ and $E_2$ respectively.
In particular, $\Delta_1$ and $\Delta_2$ can be chosen to be disjoint.
See Figure \ref{fig4}.
Hence $r'(E_1)$ and $r'(E_2)$ are disjoint or isotopic.

\begin{figure}[ht!]
\begin{center}
\includegraphics[width=7.5cm]{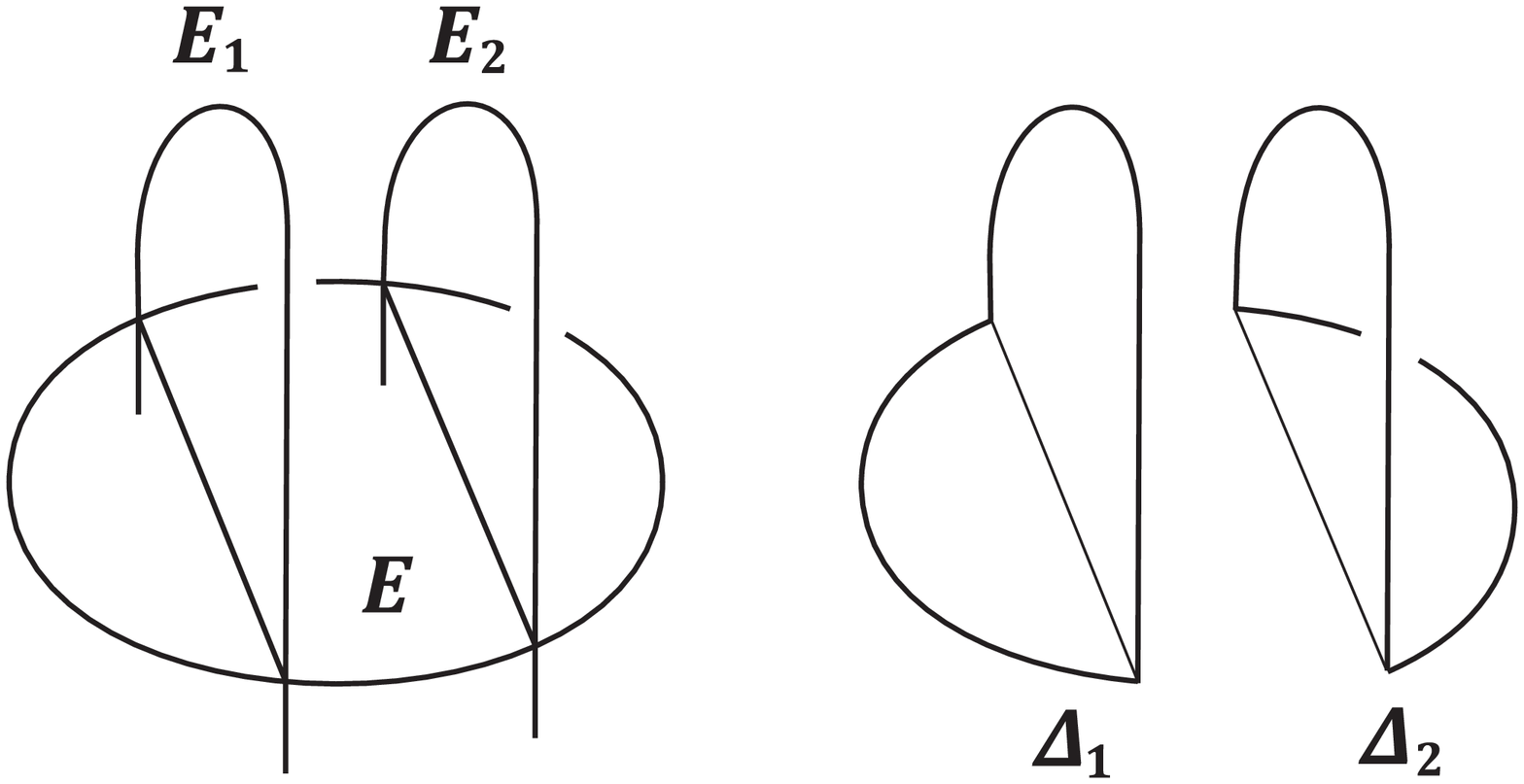}
\caption{}\label{fig4}
\end{center}
\end{figure}

Case $6$. Both of two disjoint disks $D_1$ and $D_2$ are of type $(4)$.

Since $D_1$ and $D_2$ are disjoint, $r'(D_1)$ and $r'(D_2)$ are disjoint or isotopic.
\end{proof}

Because higher dimensional simplices of $\mathcal{D}_{F_{n+1}}$
are determined by its $1$-skeleton,
the map $r'$ can be extended to all of $\mathcal{D}_{F_{n+1}}$.
It is obvious that $r'$ is identity on $S^n$.
Thus $r':\mathcal{D}_{F_{n+1}}\rightarrow S^n$ is a retraction.
So $\pi_n(\mathcal{D}_{F_{n+1}})$ is non-trivial and
$F_{n+1}$ is a topologically minimal surface of index at most $n+1$.

For $n\ge 2$, $F_n$ is not strongly irreducible,
so the topological index of $F_n$ is at least two.
It is interesting to see that $F_n$ is isotopic to a surface
obtained from $F$ or $F_1$ by stabilizations. See Figure $5$.

\begin{figure}[ht!]
\begin{center}
\includegraphics[width=9cm]{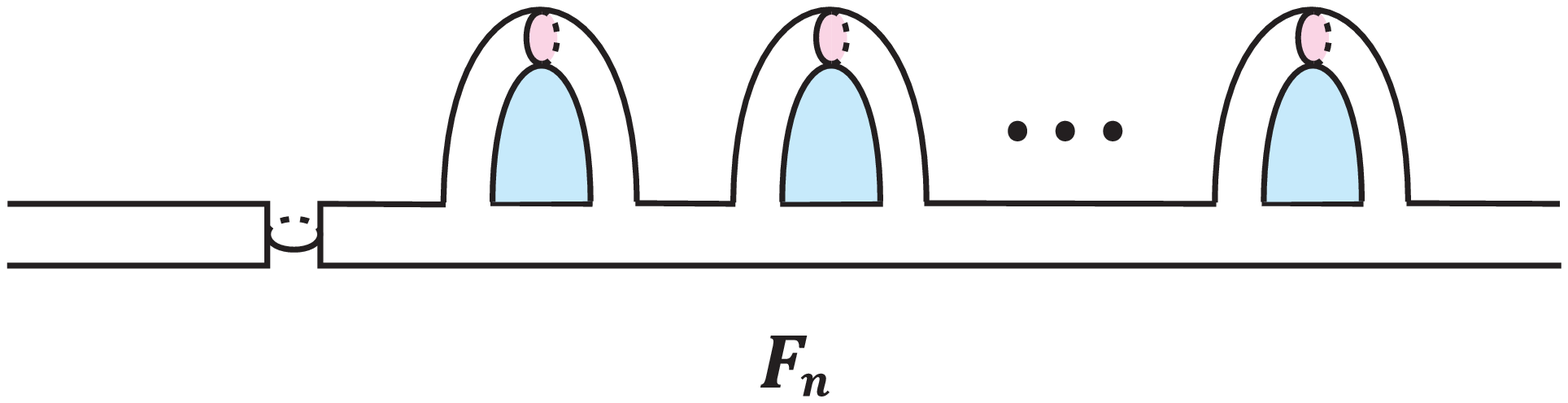}
\caption{}\label{fig5}
\end{center}
\end{figure}

\begin{question}
What is the topological index of $F_n$ for $n\ge 2$?
\end{question}


\end{document}